\documentclass{amsart}
\usepackage{latexsym,amssymb,amsmath,amsthm,amscd,graphicx,ascmac,}
\usepackage{longtable}
\usepackage{enumerate}
\usepackage{color}

\numberwithin{equation}{section}

\theoremstyle{plain}
\newtheorem{thm}{Theorem}[section]
\newtheorem*{thm*}{Theorem}

\theoremstyle{definition}

\theoremstyle{remark}

\newtheorem*{xrem}{Remark}

\newcommand{\cD}{{\mathcal D}}

\newcommand{\cF}{{\mathcal F}}
\newcommand{\cG}{{\mathcal G}}

\newcommand{\cM}{{\mathcal M}}

\newcommand{\cQ}{{\mathcal Q}}

\newcommand{\cS}{{\mathcal S}}

\newcommand{\kI}{{\mathfrak I}}

\newcommand{\kM}{{\mathfrak M}}

\newcommand{\N}{{\mathbb N}}

\newcommand{\R}{{\mathbb R}}

\newcommand{\Z}{{\mathbb Z}}

\def\al{\alpha}
\def\bt{\beta}

\def\tht{\theta}

\def\0{\emptyset}
\def\1{{\bf 1}}
\def\6{\partial}
\def\8{\infty}

\def\lt{\left}
\def\rt{\right}

\def\ds{\displaystyle}

\newcommand{\iii}[1]{{\left\vert\kern-0.25ex\left\vert\kern-0.25ex\left\vert #1 
    \right\vert\kern-0.25ex\right\vert\kern-0.25ex\right\vert}}

\def\Xint#1{\mathchoice
 {\XXint\displaystyle\textstyle{#1}}%
 {\XXint\textstyle\scriptstyle{#1}}%
 {\XXint\scriptstyle\scriptscriptstyle{#1}}%
 {\XXint\scriptscriptstyle\scriptscriptstyle{#1}}%
 \!\int}
 \def\XXint#1#2#3{{\setbox0=\hbox{$#1{#2#3}{\int}$}
 \vcenter{\hbox{$#2#3$}}\kern-.5\wd0}}

\def\fint{\Xint-}

\begin{document}

\title{The Kerman-Sawyer trace theorem for product Morrey spaces}

\author[N.~Hatano]{Naoya Hatano}
\address{
Graduate School of Information Science and Technology, The University of Osaka, 1-5, Yamadaoka, Suita-shi, Osaka 565-0871, Japan
}
\email{n.hatano.chuo@gmail.com}

\author[R.~Kawasumi]{Ryota Kawasumi}
\address{
Center for Mathematics and Data Science, Gunma University, 
4-2 Aramaki-machi, Maebashi City, Gunma 371-8510, Japan 
}
\email{r-kawasumi@gunma-u.ac.jp}

\author[H.~Saito]{Hiroki Saito}
\address{
College of Science and Technology, Nihon University,
Narashinodai 7-24-1, Funabashi City, Chiba, 274-8501, Japan
}
\email{saitou.hiroki@nihon-u.ac.jp}

\author[H.~Tanaka]{Hitoshi Tanaka}
\address{
Research and Support Center on Higher Education for the hearing and Visually Impaired, 
National University Corporation Tsukuba University of Technology,
Kasuga 4-12-7, Tsukuba City, Ibaraki, 305-8521 Japan
}
\email{htanaka@k.tsukuba-tech.ac.jp}

\thanks{}

\subjclass[2010]{Primary 42B25; Secondary 26D10, 42B35.}

\keywords{
Adams trace inequality;
Kerman-Sawyer trace inequality;
multilinear fractional integral operators;
parallel corona decomposition;
product Morrey spaces.
}

\date{}

\begin{abstract}
By using parallel corona decomposition,
the Kerman-Sawyer trace theorem is extended
from Lebesgue spaces to \textit{product Morrey spaces}.
By discretizing the multilinear fractional integral operator based on dyadic analysis, the framework of \textit{product Morrey spaces} naturally arises 
in the course of estimating the operator.
Within this natural setting, 
by establishing Sawyer-type testing estimates 
(to the setting of measures), 
we obtain an extension of 
the Kerman-Sawyer trace theorem.
The classical approach to the 
Kerman-Sawyer trace theorem typically relies on 
a reduction to Carleson's embedding theorem.
In contrast, 
in this paper we employ a parallel corona decomposition,
which allows us to overcome the difficulties inherent in the multilinear setting
and to provide a transparent and streamlined proof.
By incorporating recent developments in the theory of weights, 
this work clarifies the relationship between trace inequalities and Morrey spaces and 
contributes to a deeper understanding of these topics.
\end{abstract}

\maketitle

\section{Introduction}\label{sec1}
In this paper we investigate the Kerman-Sawyer trace theorem for product Morrey spaces.
Let $\R^n$ be the classical $n$-dimensional Euclidean space. 
The \textit{fractional integral operator}
$I_{\al}$, $0<\al<n$, 
is defined by
\[
I_{\al}f(x)
:=
\int_{\R^n}\frac{f(y)}{|x-y|^{n-\al}}\,{\rm d}y,
\quad x\in\R^n,
\]
and the \textit{fractional maximal operator} 
$M_{\al}$, $0\le\al<n$, 
is defined by 
\[
M_{\al}f(x)
:=
\sup_{x\in Q\in\cQ}
\ell_{Q}^{\al-n}
\int_{Q}|f|\,{\rm d}y,
\quad x\in\R^n.
\]
(Some notation and definitions are listed in the last of this section.)
When $\al=0$,
we simply write $M_0=M$ 
which is the classical Hardy-Littlewood maximal operator.

Let $0<p\le p_0<\8$. 
For an $L^p$ locally integrable function $f$ on $\R^n$, 
we set
\[
\|f\|_{\cM^{p,\,p_0}}
:=
\sup_{Q\in\cQ}
|Q|^{1/p_0}
\lt(\fint_{Q}|f|^p\,{\rm d}x\rt)^{1/p}.
\]
We will call the (classical) \textit{Morrey space} 
$\cM^{p,\,p_0}(\R^n)=\cM^{p,\,p_0}$ 
the subset of all $L^p$ locally integrable functions $f$ on $\R^n$ for which 
$\|f\|_{\cM^{p,\,p_0}}$ 
is finite. 
Applying H\"{o}lder's inequality,
we see that
\[
\|f\|_{\cM^{p_1,\,p_0}}
\ge
\|f\|_{\cM^{p_2,\,p_0}}
\quad\text{for all}\quad
p_0\ge p_1\ge p_2>0.
\]
This tells us that
\[
L^{p_0}=\cM^{p_0,\,p_0}
\subset\cM^{p_1,\,p_0}
\subset\cM^{p_2,\,p_0}
\quad\text{for all}\quad
p_0\ge p_1\ge p_2>0.
\]
Morrey spaces, 
introduced by C.~Morrey to study regularity questions
arising in the Calculus of Variations, 
describe local regularity more precisely than Lebesgue spaces 
and are widely used not only in harmonic analysis 
but also in partial differential equations 
(cf.~\cite{GT}).

The (well-known) Hardy-Littlewood-Sobolev theorem asserts that 
the Lebesgue space norm inequality
\[
\|I_{\al}f\|_{L^q(\R^n)}
\lesssim
\|f\|_{L^p(\R^n)}
\]
holds when 
$1<p<\frac{n}{\al}$ 
and $q=\frac{n}{n-\al p}p$.

In \cite{CF}, 
Chiarenza and Frasca showed that 
the Morrey space norm inequality
\[
\|I_{\al}f\|_{\cM^{q,\,q_0}(\R^n)}
\lesssim
\|f\|_{\cM^{p,\,p_0}(\R^n)}
\]
holds when 
$1<p\le p_0<\frac{n}{\al}$, 
$q=\frac{n}{n-\al p_0}p$ 
and 
$q_0=\frac{n}{n-\al p_0}p_0$.

Suppose that $\mu$ is a~nonnegative Radon measure on $\R^n$. 
Given any $\bt \in (0, n]$, define
\[
\|\mu\|_{\bt}
:=
\sup_{Q\in\cQ}
\frac{\mu(Q)}{\ell_{Q}^{\bt}}.
\]
Applying the Marcinkiewicz interpolation method, 
Adams in \cite{Ad} proved that 
the trace inequality 
\[
\|I_{\al}f\|_{L^q({\rm d}\mu)}
\lesssim
\|\mu\|_{n-\bt p}^{1/q}
\|f\|_{L^p({\rm d}x)}
\]
holds when
$1<p<\frac{n}{\al}$,
$0\le\bt<\al<n$ and 
$q=\frac{n-\bt p}{n-\al p}p$.

The defficulties arise in the diagonal case 
$\bt=\al$ and $q=p$.
One is now led to consider, 
for a~nonnegative Radon measure $\mu$ on $\R^n$, 
the trace inequality
\begin{equation}\label{1.1}
\|I_{\al}f\|_{L^p({\rm d}\mu)}
\lesssim A_0
\|f\|_{L^p({\rm d}x)},
\quad 1<p<\frac{n}{\al}.
\end{equation}
This inequality was studied by many authors 
(see \cite[Introduction]{KeSa}) 
and Kerman and Sawyer established the following.

\begin{quote}
The trace inequality \eqref{1.1} holds 
if and only if 
\begin{equation}\label{1.2}
A_0
=
\sup_{Q\in\cQ}
\lt(
\frac{(M_{\al}[\mu\1_{Q}])^{p'}(Q)}{\mu(Q)}
\rt)^{1/p'}<\8.
\end{equation}
\end{quote}

We show that 
\begin{equation}\label{1.3}
\|\mu\|_{n-\al p}^{1/p}
=
\sup_{Q\in\cQ}
\ell_{Q}^{\al}
\lt(\frac{\mu(Q)}{|Q|}\rt)^{1/p}
\le
\sup_{Q\in\cQ}
\lt(
\frac{(M_{\al}[\mu\1_{Q}])^{p'}(Q)}{\mu(Q)}
\rt)^{1/p'}.
\end{equation}
Indeed, by using $1/p+1/p'=1$
\begin{align*}
\ell_{Q}^{\al}
\lt(\frac{\mu(Q)}{|Q|}\rt)^{1/p}
&=
\ell_{Q}^{\al}
\lt(\frac{\mu(Q)}{|Q|}\rt)^{1/p}
\lt(\frac{\mu(Q)}{|Q|}\rt)^{1/p'}
\lt(\frac{|Q|}{\mu(Q)}\rt)^{1/p'}
\\ &=
\ell_{Q}^{\al}
\lt(\frac{\mu(Q)}{|Q|}\rt)
\lt(\frac{|Q|}{\mu(Q)}\rt)^{1/p'}
=
\lt[\lt(
\ell_{Q}^{\al}
\lt(\frac{\mu(Q)}{|Q|}\rt)
\rt)^{p'}
\lt(\frac{|Q|}{\mu(Q)}\rt)
\rt]^{1/p'}
\\ &=
\lt(\frac
{|Q|\lt(\ell_{Q}^{\al}(\mu(Q)/|Q|)\rt)^{p'}}
{\mu(Q)}
\rt)^{1/p'}
\le
\lt(
\frac{(M_{\al}[\mu\1_{Q}])^{p'}(Q)}{\mu(Q)}
\rt)^{1/p'},
\end{align*}
which yields our desired inequality \eqref{1.3}.
In \cite{SST}, 
by using an approximation of Cantor set,
it is verified that 
the reverse inequality of \eqref{1.3} can not hold.

The purpose of this paper is to extend this Kerman-Sawyer trace theorem 
to product Morrey spaces.
By discretizing the multilinear fractional integral operator based on dyadic analysis, 
the framework of product Morrey spaces naturally arises 
in the course of estimating the operator.

Let $m\in\N$. 
For any vector-valued function
\[
\vec{f}
:=
(f_1,\ldots,f_m),
\]
with each $f_i$ being a locally integrable function in $\R^n$, 
and for any $x\in\R^n$, 
the \textit{$m$-linear fractional integral operator}
$\kI_{\al}$, $0<\al<mn$, 
is defined by
\[
\kI_\al(\vec{f})(x)
:=
\int_{(\R^n)^m}
\frac
{\prod_{i=1}^mf_i(y_i)}
{(|x-y_1|+|x-y_2|+\cdots+|x-y_m|)^{mn-\al}}
\,{\rm d}y_1{\rm d}y_2\cdots{\rm d}y_m.
\]
Given a~vector-valued exponent 
$\vec{P}:=(p_1,\ldots,p_m)$ 
with 
$p_1,\ldots,p_m\in(1, \8)$, let 
$1/p=1/p_1+\cdots+1/p_m$ and 
let $0<p\le p_0<\8$.
The \textit{product Morrey space} 
$\cM^{\vec{P},\,p_0}((\R^n)^m)$
is defined as the set of all 
$\vec{f}=(f_1,\ldots,f_m)$ 
satisfying
\begin{align*}
\|\vec{f}\|_{\cM^{\vec{P},\,p_0}((\R^n)^m)}
&:=
\sup_{Q\in\cQ}
|Q|^{1/p_0}
\prod_{i=1}^m
\lt(\fint_{Q}|f_i|^{p_i}\,{\rm d}y\rt)^{1/p_i}
\\ &=
\sup_{Q\in\cQ}
|Q|^{1/p_0-1/p}
\prod_{i=1}^m
\lt(\int_{Q}|f_i|^{p_i}\,{\rm d}y\rt)^{1/p_i}
<\8.
\end{align*}
Let $\mu$ be a~nonnegative Radon measure on $\R^n$. 
For $0<q\le q_0<\8$,
the \textit{Radon-Morrey space} 
$\cM_{\mu}^{q,\,q_0}(\R^n)$
consists of all $\mu$-measurable functions $f$ on $\R^n$ 
such that
\[
\|f\|_{\cM_{\mu}^{q,\,q_0}(\R^n)}
:=
\sup_{Q\in\cQ}
|Q|^{1/q_0-1/q}
\lt(\int_{Q}|f|^q\,{\rm d}\mu\rt)^{1/q}
<\8.
\]

In \cite{HKST}, 
the following theorems were verified 
(Theorems \ref{thm1.1} and \ref{thm1.2}).

\begin{thm}{\rm(\cite[Theorem 1.1]{HKST})}\label{thm1.1}
Let $\mu$ be a nonnegative Radon measure on $\R^n$.
Let 
$\vec{P}:=(p_1,\ldots,p_m)$ with 
$p_1,\ldots,p_m\in(1, \8)$ and 
$1/p=1/p_1+\cdots+1/p_m$.
Assume that 
$1<p\le p_0<\frac{n}{\al}$ and 
$0\le\bt<\al<n$. Set 
$\tht=\frac{n-\bt p_0}{n-\al p_0}$,
$q=\tht p$ and $q_0=\tht p_0$.
Then we have
\[
\|\kI_\al(\vec{f})\|_{\cM_{\mu}^{q,\,q_0}(\R^n)}
\lesssim
\|\mu\|_{n-\bt p}^{1/q}
\|\vec{f}\|_{\cM^{\vec{P},\,p_0}((\R^n)^m)}.
\]
\end{thm}

\begin{thm}{\rm(\cite[Theorem 1.2]{HKST})}\label{thm1.2}
Let $\mu$ be a~nonnegative Radon measure on $\R^n$. 
Let 
$\vec{P}:=(p_1,\ldots,p_m)$ with 
$p_1,\ldots,p_m\in(1, \8)$ and 
$1/p=1/p_1+\cdots+1/p_m$.
Assume that 
$0<p\le p_0<\frac{n}{\al}$, 
$0<p\le 1$ and 
$0<\bt\le\al<mn$. Set 
$\tht=\frac{n-\bt p_0}{n-\al p_0}$,
$q=\tht p$ and $q_0=\tht p_0$.
Then we have
\[
\|\kI_\al(\vec{f})\|_{\cM_{\mu}^{q,\,q_0}(\R^n)}
\lesssim
\|\mu\|_{n-\bt p}^{1/q}
\|\vec{f}\|_{\cM^{\vec{P},\,p_0}((\R^n)^m)}.
\]
\end{thm}

The following is our main theorem 
(Theorem \ref{thm1.3}), 
whose proofs are the focus of this paper.

\begin{thm}\label{thm1.3}
Let $\mu$ be a~nonnegative Radon measure on $\R^n$. 
Let 
$\vec{P}:=(p_1,\ldots,p_m)$ with 
$p_1,\ldots,p_m\in(1, \8)$ and 
$1/p=1/p_1+\cdots+1/p_m$.
Assume that 
$1<p\le p_0<\frac{n}{\al}$ 
and $0<\al<n$. Then, 
the inequality
\[
\|\kI_\al(\vec{f})\|_{\cM_{\mu}^{p,\,p_0}(\R^n)}
\lesssim A_0
\|\vec{f}\|_{\cM^{\vec{P},\,p_0}((\R^n)^m)}
\]
holds if 
\[
A_0
=
\sup_{Q\in\cQ}
\lt(
\frac{(M_{\al}[\mu\1_{Q}])^{p'}(Q)}{\mu(Q)}
\rt)^{1/p'}
<\8.
\]
\end{thm}

\begin{xrem}
The $m$-linear fractional integral operator $\kI_{\al}$
and the $m$-sublinear fractional maximal operator $\kM_{\al}$ 
are studied on weighted Morrey type spaces 
in \cite{GM1,GM2,ISST1,ISST2,ISST3}.
\end{xrem}

We have used (and will use) the following notation.

\begin{enumerate}
\item
Denote by $\cQ=\cQ(\R^n)$ 
the family of all cubes in $\R^n$ with sides parallel to the axes. 
Given a~cube $Q\in\cQ$, 
denote by $c_{Q}$ and $\ell_{Q}$ its center and its side length of $Q$, respectively,
and $|Q|$ stands for the volume of $Q$.
\item
We define the set of all dyadic cubes in $\R^n$ by
\[
\cD=\cD(\R^n)
:=
\{2^{-k}(m+[0,1)^n):\,
k\in\Z,\,m\in\Z^n\}.
\]
That $\cD$ satisfies the following 
\textit{nested property}:
\begin{equation}\label{NestedProperty}
P,Q\in\cD 
\,\longrightarrow\,
P\cap Q\in\{P,Q,\0\}.
\end{equation}
\item
Given $Q\in\cD$ and 
$\cG\subset\cD$, we write
\[
\cG|_{Q}
:=
\{Q'\in\cG:\,Q'\subseteq Q\},
\]
that is, 
the restriction to $Q$ of $\cG$.
\item
We say that a family 
$\cS\subset\cD$ is \textit{sparse} 
if for every $S\in\cS$, 
there exists a measurable set 
$E_{\cS}(S)\subset S$ such that 
$2|E_{\cS}(S)|\ge|S|$, 
and the sets 
$\{E_{\cS}(S):\,S\in\cS\}$ 
are pairwise disjoint. 
\item
Given a measurable set $E\subset\R^n$, 
$\1_{E}$ denotes the characteristic function of $E$. 
\item
The barred integral 
$\fint_{S}f(y)\,{\rm d}y$ 
stands for the usual integral average of $f$ over the set $S$.
\item
Given $1<p<\8$, 
$p'=p/(p-1)$ 
denotes the conjugate exponent number of $p$. 
\item
A~\textit{weight} means a~nonnegative locally integrable function on $\R^n$.
For the weight $w$ and cube $Q\in\cQ$,
we write 
$w(Q)=\int_{Q}w\,{\rm d}x$.
\item
The letter $C$ will be used for constants that may change from one occurrence to another.
Constants with subscripts, such as $C_1$, $C_2$, do not change in different occurrences.
By $A\approx B$ we mean that 
$c^{-1}B\le A\le cB$ 
with some positive finite constant $c$ independent of appropriate quantities. 
We write $X\lesssim Y$, $Y\gtrsim X$ 
if there is a independent constant $c$ such that $X \le cY$. 
\end{enumerate}

\section{The dyadic analysis}\label{sec2}
Theorem \ref{thm1.3} 
can be proved through dyadic analysis.

Let $m\in\N$ and $\al\in[0, mn)$. 
For any vector-valued function
$\vec{f}:=(f_1,\ldots,f_m)$
with each $f_i\ge 0$ being a~locally integrable function on $\R^n$,
it is known that the operator $\kI_{\al}$
admits the following discretization:
\begin{align*}
\kI_{\al}(\vec{f})
&\lesssim
\sum_{Q\in\cD}
\ell_{Q}^{\al}
\prod_{i=1}^m
\frac{1}{|Q|}\int_{3Q}f_i\,{\rm d}x\,\1_{Q}.
\end{align*}
where 
$3Q$ denotes the cube with the same center as $Q$ 
with triple side length. 
This fact is proved in 
\cite[Section 4]{Moen09} and 
\cite[Lemma 2.1]{ISST3}.
Therefore, 
through the dyadic grid argument,
in order to prove 
Theorem \ref{thm1.3}, 
it suffices to establish 
the estimates for the following 
dyadic discrete-type operator.
Define
\begin{align*}
\kI_{\al}^{\cD}(\vec{f})
&:=
\sum_{Q\in\cD}
\ell_{Q}^{\al}
\prod_{i=1}^m
\fint_{Q}f_i\,{\rm d}x\,\1_{Q}
\\ &=
\sum_{Q\in\cD}
\ell_{Q}^{\al-mn}
\prod_{i=1}^m
\int_{Q}f_i\,{\rm d}x\,\1_{Q}.
\end{align*}

\subsection{The parallel corona decomposition}\label{ssec2.2}
In what follows, 
we use parallel corona decomposition 
(cf. the book \cite{ST} 
or the papers 
\cite{HHL,Hy,Ta1,Ta2,Ta3}).

\begin{thm}\label{thm2.1}
Let $\mu$ be a~nonnegative Radon measure on $\R^n$. 
Let $K:\cD\to[0, \8)$ be a~map. 
Let 
$p,p_1,\ldots,p_m\in(1, \8)$ with 
$1/p=1/p_1+\cdots+1/p_m$. 
Then, for any $Q_0\in\cD$, 
the multilinear embedding inequality
\begin{equation}\label{2.1}
\sum_{Q\in\cD|_{Q_0}}
K(Q)\prod_{i=1}^m
\int_{Q}f_i\,{\rm d}x
\int_{Q}g\,{\rm d}\mu
\lesssim A_0
\prod_{i=1}^m
\|f_i\1_{Q_0}\|_{L^{p_i}({\rm d}x)}
\|g\1_{Q_0}\|_{L^{p'}({\rm d}\mu)}
\end{equation}
holds for the nonnegative functions 
$f_1,\ldots,f_m\in L_{{\rm loc}}^1({\rm d}x)$ 
and 
$g\in L_{{\rm loc}}^1({\rm d}\mu)$. 
Here, 
$A_0=\sup_{Q\in\cD}A_{Q}$ 
and 
\[
A_{Q}
=
\max\lt(\begin{array}{l}\ds
\frac1{\mu(Q)^{1/p'}}
\sup_{\cS}
\lt(
\sum_{S\in\cS|_{Q}}
\lt(
\frac1{|S|^{1/p}}
\sum_{Q'\in\cD|_{S}}
K(Q')|Q'|^m\mu(Q')
\rt)^{p'}\rt)^{1/p'},
\\ \ds
\frac1{|Q|^{1/p}}
\sup_{\cS_{\mu}}
\lt(
\sum_{S\in\cS_{\mu}|_{Q}}
\lt(
\frac1{\mu(S)^{1/p'}}
\sum_{Q'\in\cD|_{S}}
K(Q')|Q'|^m\mu(Q')
\rt)^p\rt)^{1/p}
\end{array}\rt),
\]
where the supremums are taken over 
all sparse family $\cS$ with respect to ${\rm d}x$ 
and 
all sparse family $\cS_{\mu}$ with respect to $\mu$, 
respectively.
\end{thm}

\noindent\textbf{Proof}\quad
There holds
\begin{align*}
\sum_{Q\in\cD|_{Q_0}}
K(Q)\prod_{i=1}^m
\int_{Q}f_i\,{\rm d}x
\int_{Q}g\,{\rm d}\mu
&=
\sum_{Q\in\cD|_{Q_0}}
K(Q)|Q|^m\prod_{i=1}^m
\fint_{Q}f_i\,{\rm d}x
\int_{Q}g\,{\rm d}\mu
\\ &\le
\sum_{Q\in\cD|_{Q_0}}
K(Q)|Q|^{m-1}
\int_{Q}\kM^{\cD|_{Q_0}}(\vec{f})\,{\rm d}x
\int_{Q}g\,{\rm d}\mu,
\end{align*}
where 
\[
\kM^{\cD|_{Q_0}}(\vec{f})
:=
\sup_{Q\in\cD|_{Q_0}}
\prod_{i=1}^m
\fint_{Q}f_i\,{\rm d}x\,\1_{Q}.
\]
For simplicity, we write 
$\kM^{\cD|_{Q_0}}(\vec{f})=\phi$.

It follows from multiple H\"{o}lder's inequality 
with 
$1/p=1/p_1+\cdots+1/p_m$ and 
the boundedness of dyadic Hardy-Littlewood maximal operator 
that
\begin{equation}\label{2.2}
\|\phi\1_{Q_0}\|_{L^p({\rm d}x)}
\lesssim
\prod_{i=1}^m
\|f_i\1_{Q_0}\|_{L^{p_i}({\rm d}x)}.
\end{equation}

We define the collection of principal cubes 
$\cF$ for the pair $(\phi,{\rm d}x)$ 
and $\cG$ for the pair $(g,\mu)$. 
Namely, analogously for $\cG$, 
\[
\cF
:=
\bigcup_{k=0}^{\8}\cF_k,
\]
where 
$\cF_0:=\{Q_0\}$,
\[
\cF_{k+1}
:=
\bigcup_{F\in\cF_k}
{\rm ch}_{\cF}(F)
\]
and ${\rm ch}_{\cF}(F)$ 
is defined by 
the set of all \textit{maximal} dyadic cubes $Q\subset F$ such that 
\[
\fint_{Q}\phi\,{\rm d}x
>
2\fint_{F}\phi\,{\rm d}x.
\]
Observe that
\begin{align*}
\sum_{F'\in{\rm ch}_{\cF}(F)}
|F'|
&\le
\lt(2\fint_{F}\phi\,{\rm d}x\rt)^{-1}
\sum_{F'\in{\rm ch}_{\cF}(F)}
\int_{F'}\phi\,{\rm d}x
\le
\lt(2\fint_{F}\phi\,{\rm d}x\rt)^{-1}
\int_{F}\phi\,{\rm d}x
=
\frac{|F|}{2},
\end{align*}
which implies 
\begin{equation}\label{2.3}
|E_{\cF}(F)|
:=
\lt|
F\setminus\bigcup_{F'\in{\rm ch}_{\cF}(F)}
F'
\rt|
\ge
\frac{|F|}{2},
\end{equation}
where the sets 
$E_{\cF}(F)$, $F\in\cF$, 
are pairwise disjoint. 

We further define the stopping parents, 
for $Q\in\cD|_{Q_0}$, 
\[
\begin{cases}\ds
\pi_{\cF}(Q)
:=
\min\{F\supset Q:\,F\in\cF\},
\\ \ds
\pi_{\cG}(Q)
:=
\min\{G\supset Q:\,G\in\cG\},
\\ \ds
\pi(Q)
:=
(\pi_{\cF}(Q),\pi_{\cG}(Q)).
\end{cases}
\]
Then we can rewrite the series in \eqref{2.1} as follows:
\begin{align*}
\sum_{Q\in\cD|_{Q_0}}
=
\sum_{\substack{F\in\cF, \\ G\in\cG}}
\sum_{\substack{Q: \\ \pi(Q)=(F,G)}}
\le
\sum_{F\in\cF}
\sum_{G\in\cG|_{F}}
\sum_{\substack{Q: \\ \pi(Q)=(F,G)}}
+
\sum_{G\in\cG}
\sum_{F\in\cF|_{G}}
\sum_{\substack{Q: \\ \pi(Q)=(F,G)}},
\end{align*}
where we have used 
\eqref{NestedProperty}:
if $P,Q\in\cD$ then 
$P\cap Q\in\{P,Q,\0\}$. 
Since the proof can be done in a completely symmetric way, 
we shall concentrate ourselves on the second case only.

We notice that
\begin{equation}\label{2.4}
\begin{cases}\ds
\fint_{Q}\phi\,{\rm d}x
\le
2\fint_{F}\phi\,{\rm d}x
\quad\text{when}\quad
\pi_{\cF}(Q)=F,
\\ \ds
\fint_{Q}g\,{\rm d}\mu
\le
2\fint_{G}g\,{\rm d}\mu
\quad\text{when}\quad
\pi_{\cG}(Q)=G.
\end{cases}
\end{equation}
Here, 
$\fint_{Q}g\,{\rm d}\mu
=
\mu(Q)^{-1}\int_{Q}g\,{\rm d}\mu.$

We estimate 
\begin{equation}\label{2.5}
{\rm(i)}
:=
\sum_{G\in\cG}
\sum_{F\in\cF|_{G}}
\sum_{\substack{Q: \\ \pi(Q)=(F,G)}}
K(Q)|Q|^{m-1}
\int_{Q}\phi\,{\rm d}x
\int_{Q}g\,{\rm d}\mu.
\end{equation}

Fix $G\in\cG$. 
It follows from \eqref{2.4} that
\begin{align*}
\lefteqn{
\sum_{F\in\cF|_{G}}
\sum_{\substack{Q: \\ \pi(Q)=(F,G)}}
K(Q)|Q|^{m-1}
\int_{Q}\phi\,{\rm d}x
\int_{Q}g\,{\rm d}\mu
}\\ &\le 2
\fint_{G}g\,{\rm d}\mu
\sum_{F\in\cF|_{G}}
\sum_{\substack{Q: \\ \pi(Q)=(F,G)}}
K(Q)|Q|^{m-1}\mu(Q)
\int_{Q}\phi\,{\rm d}x.
\end{align*}

We need two observations. 
Suppose that 
$\pi(Q)=(F,G)$, 
$F\in\cF|_{G}$. 
If $G'\in{\rm ch}_{\cG}(G)$ 
satisfies $G'\subset Q$, then, 
by the definition of $\pi(Q)$, 
we must have 
\begin{equation}\label{2.6}
\pi_{\cG}\lt(\pi_{\cF}(G')\rt)
=
\begin{cases}\ds
G,\quad G'\notin\cF|_{G},
\\ \ds
G',\quad G'\in\cF|_{G}.
\end{cases}
\end{equation}
By this observation, we define 
\[
{\rm ch}_{\cG}^*(G)
:=
\lt\{G'\in{\rm ch}_{\cG}(G):\,\text{
$G'$ satisfies \eqref{2.6}
}\rt\}.
\]
We further observe that, 
when $G'\in{\rm ch}_{\cG}^*(G)$, 
we can regard $\phi$ as a constant on $G'$ in the above integrals, 
that is, 
we can replace $\phi$ by $\phi^{G}$ in the above integrals,
where
\[
\phi^{G}
:=
\phi\1_{E_{\cG}(G)}
+
\sum_{G'\in{\rm ch}_{\cG}^*(G)}
\fint_{G'}\phi\,{\rm d}x\,\1_{G'}.
\]

It follows from \eqref{2.4} and 
H\"{o}lder's inequality with 
$1/p+1/p'=1$ that
\begin{align*}
\lefteqn{
\sum_{F\in\cF|_{G}}
\sum_{\substack{Q: \\ \pi(Q)=(F,G)}}
K(Q)|Q|^{m-1}\mu(Q)
\int_{Q}\phi^{G}\,{\rm d}x
}\\ &\le 2
\sum_{F\in\cF|_{G}}
\fint_{F}\phi^{G}\,{\rm d}x
\sum_{\substack{Q: \\ \pi(Q)=(F,G)}}
K(Q)|Q|^m\mu(Q)
\\ &\le 2
\lt[
\sum_{F\in\cF|_{G}}
\lt(\fint_{F}\phi^{G}\,{\rm d}x\rt)^p
|F|
\rt]^{1/p}
\lt[
\sum_{F\in\cF|_{G}}
\lt(
\frac1{|F|^{1/p}}
\sum_{Q\in\cD|_{F}}
K(Q)|Q|^m\mu(Q)
\rt)^{p'}
\rt]^{1/p'}.
\end{align*}
By the fact that 
$|F|\le 2|E_{\cF}(F)|$ and 
$E_{\cF}(F)$, $F\in\cF|_{G}$, 
are pairwise disjoint,
\begin{align*}
\lt[
\sum_{F\in\cF|_{G}}
\lt(\fint_{F}\phi^{G}\,{\rm d}x\rt)^p
|F|
\rt]^{1/p}
&\lesssim
\lt[
\sum_{F\in\cF|_{G}}
\lt(\fint_{F}\phi^{G}\,{\rm d}x\rt)^p
|E_{\cF}(F)|
\rt]^{1/p}
\\ &\le
\lt[
\int_{G}(M\phi^{G})^p\,{\rm d}x
\rt]^{1/p}
\lesssim
\|\phi^{G}\|_{L^p({\rm d}x)}.
\end{align*}

Setting
\[
A_{G}
=
\frac1{\mu(G)^{1/p'}}
\lt[
\sum_{F\in\cF|_{G}}
\lt(
\frac1{|F|^{1/p}}
\sum_{Q\in\cD|_{F}}
K(Q)|Q|^m\mu(Q)
\rt)^{p'}\rt]^{1/p'},
\]
we conclude 
\[
{\rm(i)}
\lesssim
\sup_{G\in\cG}A_{G}
\sum_{G\in\cG}
\|\phi^{G}\|_{L^p({\rm d}x)}
\fint_{G}g\,{\rm d}\mu
\mu(G)^{1/p'}.
\]
From H\"{o}lder's inequality,
\begin{align*}
\lefteqn{
\sum_{G\in\cG}
\|\phi^{G}\|_{L^p({\rm d}x)}
\fint_{G}g\,{\rm d}\mu
\mu(G)^{1/p'}
}\\ &\le
\lt(
\sum_{G\in\cG}
\|\phi^{G}\|_{L^p({\rm d}x)}^p
\rt)^{1/p}
\lt(
\sum_{G\in\cG}
\lt(\fint_{G}g\,{\rm d}\mu\rt)^{p'}
\mu(G)
\rt)^{1/p'}
=:
{\rm(i_1)}\times{\rm(i_2)}.
\end{align*}

For ${\rm(i_2)}$, using 
$\mu(G)\le 2\mu(E_{\cG}(G))$,
the fact that 
\[
\fint_{G}g\,{\rm d}\mu
\le
\inf_{y\in G}M_{\mu}g(y)
\]
and the disjointness of the sets 
$E_{\cG}(G)$, we have that
\begin{align*}
{\rm(i_2)}
&\lesssim
\lt(
\sum_{G\in\cG}
\int_{E_{\cG}(G)}(M_{\mu}[g\1_{Q_0}])^{p'}\,{\rm d}\mu
\rt)^{1/p'}
\le
\|M_{\mu}[g\1_{Q_0}]\|_{L^{p'}({\rm d}\mu)}
\lesssim
\|g\1_{Q_0}\|_{L^{p'}({\rm d}\mu)}.
\end{align*}
Here, 
$M_{\mu}$ is the dyadic Hardy-Littlewood maximal operator and 
we have used the $L^{p'}({\rm d}\mu)$-boundedness of $M_{\mu}$. 

It remains to estimate ${\rm(i_1)}$. 
It follows that 
\[
{\rm(i_1)}^p
=
\sum_{G\in\cG}
\int_{E_{\cG}(G)}\phi^p\,{\rm d}x
+
\sum_{G\in\cG}
\sum_{G'\in{\rm ch}_{\cG}^*(G)}
\lt(\fint_{G'}\phi\,{\rm d}x\rt)^p
|G'|.
\]
By the pairwise disjointness of the sets $E_{\cG}(G)$, 
it is immediate that
\[
\sum_{G\in\cG}
\int_{E_{\cG}(G)}\phi^p\,{\rm d}x
\le
\|\phi\1_{Q_0}\|_{L^p({\rm d}x)}^p.
\]
For the remaining double sum, 
we first notice that
\begin{align*}
{\rm(ii)}
&:=
\sum_{F\in\cF}
\lt(\fint_{F}\phi\,{\rm d}x\rt)^p
|F|
\lesssim
\|M[\phi\1_{Q_0}]\|_{L^p({\rm d}x)}^p
\le
\|\phi\1_{Q_0}\|_{L^p({\rm d}x)}^p.
\end{align*}
It follows from \eqref{2.6} that 
\begin{align*}
\sum_{G\in\cG}
\sum_{\substack{
G'\in{\rm ch}_{\cG}^*(G): \\ \pi_{\cG}\lt(\pi_{\cF}(G')\rt)=G
}}
\lt(\fint_{G'}\phi\,{\rm d}x\rt)^p
|G'|
&=
\sum_{F\in\cF}
\sum_{\substack{
G'\in{\rm ch}_{\cG}^*(G): \\ \pi_{\cF}(G')=F
}}
\lt(\fint_{G'}\phi\,{\rm d}x\rt)^p
|G'|
\\ &\le 2^p
\sum_{F\in\cF}
\lt(\fint_{F}\phi\,{\rm d}x\rt)^p
\sum_{\substack{
G'\in{\rm ch}_{\cG}^*(G): \\ \pi_{\cF}(G')=F
}}
|G'|
\\ &\le 2^p
\sum_{F\in\cF}
\lt(\fint_{F}\phi\,{\rm d}x\rt)^p
|F|
=2^p{\rm(ii)}.
\end{align*}
It follows also from \eqref{2.6} that 
\begin{align*}
\sum_{G\in\cG}
\sum_{\substack{
G'\in{\rm ch}_{\cG}^*(G): \\ \pi_{\cG}\lt(\pi_{\cF}(G')\rt)=G'
}}
\lt(\fint_{G'}\phi\,{\rm d}x\rt)^p
|G'|
\le
\sum_{F\in\cF}
\lt(\fint_{F}\phi\,{\rm d}x\rt)^p
|F|
=2^p{\rm(ii)}.
\end{align*}

Altogether, we obtain 
\[
{\rm(i)}
\lesssim
\sup_{Q\in\cD}A_{Q}
\cdot
\prod_{i=1}^m
\|f_i\1_{Q_0}\|_{L^{p_i}({\rm d}x)}
\|g\1_{Q_0}\|_{L^{p'}({\rm d}\mu)}.
\]
This yields the theorem. 
\qed

\begin{thm}\label{thm2.2}
Let $\mu$ be a~nonnegative Radon measure on $\R^n$. 
Let $K:\cD\to[0, \8)$ be a~map. 
Let 
$\vec{P}:=(p_1,\ldots,p_m)$ with 
$p_1,\ldots,p_m\in(1, \8)$ and 
$1/p=1/p_1+\cdots+1/p_m$.
Let $1<p\le p_0<\8$.
For a nonnegative vector-valued function
$\vec{f}=(f_1,\ldots,f_m)$,
where each $f_i\ge 0$ 
is locally integrable on $\R^n$, 
we define
\[
\kI_{K}^{\cD}(\vec{f})
:=
\sum_{Q\in\cD}
K(Q)\prod_{i=1}^m
\int_{Q}f_i\,{\rm d}x\,\1_{Q}.
\]
Suppose that, for any $Q\in\cD$,
\begin{equation}\tag{D}
\sum_{\substack{
Q'\in\cD: \\ Q'\supsetneq Q
}}
K(Q')|Q'|^{m-1/p_0}
\lesssim
K(Q)|Q|^{m-1/p_0}.
\end{equation}
Then, the inequality
\[
\|\kI_{K}^{\cD}(\vec{f})\|_{\cM_{\mu}^{p,\,p_0}(\R^n)}
\lesssim A_0
\|\vec{f}\|_{\cM^{\vec{P},\,p_0}((\R^n)^m)}
\]
holds if 
$A_0=\sup_{Q\in\cD}A_{Q}<\8$ 
and 
\[
A_{Q}
=
\max\lt(\begin{array}{l}\ds
\frac1{\mu(Q)^{1/p'}}
\sup_{\cS}
\lt(
\sum_{S\in\cS|_{Q}}
\lt(
\frac1{|S|^{1/p}}
\sum_{Q'\in\cD|_{S}}
K(Q')|Q'|^m\mu(Q')
\rt)^{p'}\rt)^{1/p'},
\\ \ds
\frac1{|Q|^{1/p}}
\sup_{\cS_{\mu}}
\lt(
\sum_{S\in\cS_{\mu}|_{Q}}
\lt(
\frac1{\mu(S)^{1/p'}}
\sum_{Q'\in\cD|_{S}}
K(Q')|Q'|^m\mu(Q')
\rt)^p\rt)^{1/p}
\end{array}\rt),
\]
where the supremums are taken over 
all sparse family $\cS$ with respect to ${\rm d}x$ 
and 
all sparse family $\cS_{\mu}$ with respect to $\mu$, 
respectively.
\end{thm}

\noindent\textbf{Proof}\quad
Fix $Q_0\in\cD$.
We decompose
\begin{align*}
\kI_{K}^{\cD}(\vec{f})
&=
\lt(
\sum_{Q\in\cD|_{Q_0}}
+
\sum_{\substack{
Q\in\cD: \\ Q\supsetneq Q_0
}}
\rt)
K(Q)\prod_{i=1}^m
\int_{Q}f_i\,{\rm d}x\,\1_{Q}
=:
F_1+F_2.
\end{align*}

By applying Theorem \ref{thm2.1}, 
we have that, for the nonnegative function $g$,
\[
\sum_{Q\in\cD|_{Q_0}}
K(Q)\prod_{i=1}^m
\int_{Q}f_i\,{\rm d}x
\,\int_{Q}g\,{\rm d}\mu
\lesssim A_0
\prod_{i=1}^m
\|f_i\1_{Q_0}\|_{L^{p_i}({\rm d}x)}
\|g\1_{Q_0}\|_{L^{p'}({\rm d}\mu)}
\]
holds if 
$A_0=\sup_{Q\in\cD}A_{Q}<\8$.

This implies
\begin{align*}
|Q_0|^{1/p_0-1/p}
\|F_1\|_{L^p({\rm d}\mu)}
&\lesssim A_0
|Q_0|^{1/p_0-1/p}
\prod_{i=1}^m
\|f_i\1_{Q_0}\|_{L^{p_i}({\rm d}x)}
\le A_0
\|\vec{f}\|_{\cM^{\vec{P},\,p_0}((\R^n)^m)}.
\end{align*}

There holds
\begin{align*}
\lefteqn{
|Q_0|^{1/p_0-1/p}
\|F_2\|_{L^p({\rm d}\mu)}
}\\ &\le
|Q_0|^{1/p_0-1/p}
\|\1_{Q_0}\|_{L^p({\rm d}\mu)}
\sum_{\substack{
Q\in\cD: \\ Q\supsetneq Q_0
}}
K(Q)\prod_{i=1}^m
\int_{Q}f_i\,{\rm d}x
\\ &=
|Q_0|^{1/p_0}
\lt(\frac{\mu(Q_0)}{|Q_0|}\rt)^{1/p}
\sum_{\substack{
Q\in\cD: \\ Q\supsetneq Q_0
}}
K(Q)|Q|^m
\prod_{i=1}^m
\fint_{Q}f_i\,{\rm d}x
\\ &\le
|Q_0|^{1/p_0}
\lt(\frac{\mu(Q_0)}{|Q_0|}\rt)^{1/p}
\sum_{\substack{
Q\in\cD: \\ Q\supsetneq Q_0
}}
K(Q)|Q|^{m-1/p_0}
\lt[
|Q|^{1/p_0}
\prod_{i=1}^m
\lt(
\fint_{Q}f_i^{p_i}\,{\rm d}x
\rt)^{1/p_i}
\rt]
\\ &\le
\|\vec{f}\|_{\cM^{\vec{P},\,p_0}((\R^n)^m)}
|Q_0|^{1/p_0}
\lt(\frac{\mu(Q_0)}{|Q_0|}\rt)^{1/p}
\sum_{\substack{
Q\in\cD: \\ Q\supseteq Q_0
}}
K(Q)|Q|^{m-1/p_0}
\\ &\lesssim
K(Q_0)|Q_0|^m
\lt(\frac{\mu(Q_0)}{|Q_0|}\rt)^{1/p}
\|\vec{f}\|_{\cM^{\vec{P},\,p_0}((\R^n)^m)},
\end{align*}
where we have used (D).
By noticing
\[
K(Q_0)|Q_0|^m
\lt(\frac{\mu(Q_0)}{|Q_0|}\rt)^{1/p}
\le A_0,
\]
these complete the proof.
\qed

\section{Proof of Theorem \ref{thm1.3}}\label{sec3}
In what follows, 
we prove Theorem \ref{thm1.3}.
We first notice the following.

\begin{itemize}
\item
Let $1<p\le p_0<\frac{n}{\al}$, 
$0<\al<n$.
Set, for $Q\in\cD$, 
$K(Q):=\ell_{Q}^{\al-mn}$. 
The condition:
\begin{equation}\tag{D}
\sum_{\substack{
Q'\in\cD: \\ Q'\supsetneq Q
}}
K(Q')|Q'|^{m-1/p_0}
\lesssim
K(Q)|Q|^{m-1/p_0}
\end{equation}
is satisfied because then 
$\al-n/p_0<0$.
\item
Recall that, 
for a~sparse family $\cS$,
\[
A_{Q}'
:=
\frac1{\mu(Q)^{1/p'}}
\lt(
\sum_{S\in\cS|_{Q}}
\lt(
\frac1{|S|^{1/p}}
\sum_{Q'\in\cD|_{S}}
K(Q')|Q'|^m\mu(Q')
\rt)^{p'}\rt)^{1/p'}.
\]
Setting 
$K(Q'):=\ell_{Q'}^{\al-mn}$,
we have that
\begin{align*}
A_{Q}'
&\lesssim
\frac1{\mu(Q)^{1/p'}}
\lt(
\sum_{S\in\cS|_{Q}}
\lt(
\ell_{S}^{\al}
\frac{\mu(S)}{|S|^{1/p}}
\rt)^{p'}
\rt)^{1/p'}
=
\frac1{\mu(Q)^{1/p'}}
\lt(
\sum_{S\in\cS|_{Q}}
\lt(
\ell_{S}^{\al}
\frac{|S|^{1/p'}\mu(S)}{|S|}
\rt)^{p'}
\rt)^{1/p'}
\\ &=
\frac1{\mu(Q)^{1/p'}}
\lt(
\sum_{S\in\cS|_{Q}}
\lt(
\ell_{S}^{\al}
\frac{\mu(S)}{|S|}
\rt)^{p'}
|S|
\rt)^{1/p'}
\lesssim
\frac1{\mu(Q)^{1/p'}}
\lt(
\sum_{S\in\cS|_{Q}}
\lt(
\ell_{S}^{\al}
\frac{\mu(S)}{|S|}
\rt)^{p'}
|E_{\cS}(S)|
\rt)^{1/p'}
\\ &\le
\frac1{\mu(Q)^{1/p'}}
\lt(
\int_{Q}M_{\al}[\mu\1_{Q}]^{p'}\,{\rm d}x
\rt)^{1/p'}
=
\lt(
\frac{(M_{\al}[\mu\1_{Q}])^{p'}(Q)}{\mu(Q)}
\rt)^{1/p'}.
\end{align*}

Recall also that, 
for a~sparse family $\cS_{\mu}$ 
with respect to $\mu$,
\[
A_{Q}''
:=
\frac1{|Q|^{1/p}}
\lt(
\sum_{S\in\cS_{\mu}|_{Q}}
\lt(
\frac1{\mu(S)^{1/p'}}
\sum_{Q'\in\cD|_{S}}
K(Q')|Q'|^m\mu(Q')
\rt)^p\rt)^{1/p}.
\]
Setting 
$K(Q'):=\ell_{Q'}^{\al-mn}$,
we have that
\begin{align*}
A_{Q}''
&\lesssim
\frac1{|Q|^{1/p}}
\lt(
\sum_{S\in\cS_{\mu}|_{Q}}
\lt(
\ell_{S}^{\al}
\frac{\mu(S)}{\mu(S)^{1/p'}}
\rt)^p
\rt)^{1/p}
\le
\frac{\ell_{Q}^{\al}}{|Q|^{1/p}}
\lt(
\sum_{S\in\cS_{\mu}|_{Q}}
\mu(S)
\rt)^{1/p}
\\ &\lesssim
\frac{\ell_{Q}^{\al}}{|Q|^{1/p}}
\lt(
\sum_{S\in\cS_{\mu}|_{Q}}
\mu(E_{\cS_{\mu}}(S))
\rt)^{1/p}
\le
\ell_{Q}^{\al}
\lt(\frac{\mu(Q)}{|Q|}\rt)^{1/p}
\le
\lt(
\frac{(M_{\al}[\mu\1_{Q}])^{p'}(Q)}{\mu(Q)}
\rt)^{1/p'},
\end{align*}
where in the last step 
we have used \eqref{1.3}.
\end{itemize}

Thus, 
Theorem \ref{thm1.3} 
holds from 
Theorem \ref{thm2.2}.

\subsection*{Ethics approval and consent to participate}
Not applicable.
\subsection*{Consent for publication}
Not applicable.
\subsection*{Competing Interests}
The authors have no competing interests to declare that are relevant to the content of this article.
\subsection*{Author contributions}
The authors contributed equally to the correctness of this paper.
\subsection*{Funding}
The first named author is financially supported by 
the Grant-in-Aid for JSPS Fellows (No. 25KJ0222).
The second named author is supported by 
Grant-in-Aid for Research Activity Start-up (25K23330)
The third named author is supported by 
Grant-in-Aid for Scientific Research (C) (23K03171), 
the Japan Society for the Promotion of Science. 
The forth named author is supported by 
Grant-in-Aid for Scientific Research (C) (15K04918, 19K03510 and 26K06821), 
the Japan Society for the Promotion of Science.


\begin{thebibliography}{999}

\bibitem{Ad} D.~R.~Adams, 
\emph{Traces of potentials arising from translation invariant operators},
Ann. Scuola Norm. Sup. Pisa Cl. Sci. (3) \textbf{25} (1971), 203--217.

\bibitem{CF}
F.~Chiarenza and M.~Frasca, 
\emph{Morrey spaces and Hardy-Littlewood maximal function}, 
Rend. Math. Appl., \textbf{7} (1987), 273--279.

\bibitem{GM1} L.~Grafakos and A.~Meskhi, 
\emph{Sharp Olsen's inequality for multilinear Riesz potentials}, 
Trans. A.~Razmadze Math. Inst. \textbf{175} (2021), no. 3, 433--436.

\bibitem{GM2} \bysame,
\emph{On sharp Olsen's and trace inequalities for multilinear fractional integrals}, 
Potential Anal. \textbf{59} (2023), no. 3, 1039--1050.

\bibitem{GT} D.~Gilbarg and S.~N.~Trudinger,
\emph{{\it Elliptic Partial Differential Equations of Second Order}},
2nd edn. (Springer, Berlin, 1983).

\bibitem{HHL}
T.~H\"{a}nninen, T.\,P.~Hyt\"{o}nen and K.~Li, 
\emph{Two-weight $L^p$-$L^q$ bounds for positive dyadic operators: 
unified approach to $p\le q$ and $p>q$},
Potential Anal., \textbf{45} (2016), no.~3, 579--608.

\bibitem{HKST}
N.~Hatano, R.~Kawasumi, H.~Saito and H.~Tanaka, 
\emph{The Adams trace theorem for product Morrey spaces},
to appear in J. Geom. Anal., arXiv:2604.17432 (math).

\bibitem{Hy} T.\,P.~Hyt\"{o}nen,
\emph{The $A_2$ theorem: Remarks and complements},
Harmonic analysis and partial differential equations, 91--106. 
American Mathematical Society, Providence, RI, 2014. 
arXiv:1212.3840 [math.CA].

\bibitem{ISST1} T.~Iida, E.~Sato, Y.~Sawano and H.~Tanaka, 
\emph{Weighted norm inequalities for multilinear fractional operators on Morrey spaces},
Studia Math. \textbf{205} (2011), no. 2, 139--170.

\bibitem{ISST2} \bysame,
\emph{Multilinear fractional integrals on Morrey spaces}, 
Acta Math. Sin. (Engl. Ser.) \textbf{28} (2012), no. 7, 1375--1384.

\bibitem{ISST3} \bysame,
\emph{Sharp bounds for multilinear fractional integral operators on Morrey type spaces}, 
Positivity \textbf{16} (2012), no. 2, 339--358.

\bibitem{KeSa} Kerman R. and Sawyer E., 
\emph{The trace inequality and eigenvalue estimates for Schr\"{o}dinger operators},
Ann. Inst. Fourier (Grenoble), \textbf{36} (1986), 207--228.

\bibitem{Moen09}
K.~Moen,
\emph{Weighted inequalities for multilinear fractional integral operators},
Collect. Math. \textbf{60} (2009), 213--238. 

\bibitem{ST} H.~Saito and H.~Tanaka,
\emph{{\it Introduction to Potential Theory: Maximal Operators and Weights}},
(Advances in Analysis and Geometry) 
de Gruyter , 2024, 
(ISBN: 3110725967).

\bibitem{SST}
Y.~Sawano, S.~Sugano and H.~Tanaka,
\emph{Generalized fractional integral operators and fractional maximal operators in the framework of Morrey spaces },
Trans. Amer. Math. Soc.  \textbf{363}~(2011), (12), 6481--6503.

\bibitem{Ta1} H.~Tanaka, 
\emph{A~characterization of two-weight trace inequalities for positive dyadic operators in the upper triangle case},
Potential Anal., \textbf{41}~(2014), 487--499. 

\bibitem{Ta2} \bysame,
\emph{The trilinear embedding theorem}, 
Studia Math., \textbf{227} (2015), no.~3, 238--249. 

\bibitem{Ta3} \bysame,
\emph{The $n$ linear embedding theorem},
Potential Anal., \textbf{44} (2016), no.~4, 793--809.

\end{thebibliography}
\end{document}